\documentclass[11pt, a4paper]{amsart}
\usepackage{latexsym,mathrsfs,bm}
\usepackage{graphicx,color,url}
\usepackage{amssymb,mathrsfs}
\usepackage{hyperref}

\usepackage[english]{babel}
\usepackage{paralist}
\usepackage{amsthm,amsfonts,amssymb,amsmath}
\usepackage{mathtools}
\usepackage[utf8]{inputenc}

\newcommand{\Q}{\mathbb{Q}}
\newcommand{\F}{\mathbb{F}}
\newcommand{\C}{\mathbb{C}}
\newcommand{\N}{\mathbb{N}}

\newcommand{\Z}{\mathbb{Z}}

\newcommand{\es}[1]{\begin{equation}\begin{split}#1\end{split}\end{equation}}
\newcommand{\est}[1]{\begin{equation*}\begin{split}#1\end{split}\end{equation*}}

\newcommand{\nequiv}{\not\equiv}

\newcommand{\leg}[2]{\left(\frac{#1}{#2}\right)}

\renewcommand{\mod}[1]{~\pr{\textnormal{mod}~#1}}
\newtheorem*{theo*}{Theorem}
\newtheorem{theo}{Theorem}

\newtheorem*{ex}{Example}

\newtheorem{lemma}{Lemma}
\newtheorem{prop}[lemma]{Proposition}
\newtheorem{conj}{Conjecture}
\newtheorem{defin}[lemma]{Definition}

\newtheorem{corol}[lemma]{Corollary}
\newtheorem{remark}[lemma]{Remark}
\newtheorem*{rem*}{Remark}

\newcommand{\pr}[1]{\left( #1\right)}

\DeclareMathOperator{\Gal}{Gal}

\newcommand{\Res}{\operatorname*{Res}}
\newcommand{\rank}{\operatorname*{rank}}

\newcommand{\E}{\mathcal{E}}
\newcommand{\D}{D}

\let\originalleft\left
\let\originalright\right
\renewcommand{\left}{\mathopen{}\mathclose\bgroup\originalleft}
\renewcommand{\right}{\aftergroup\egroup\originalright}

\definecolor{pink}{rgb}{1,.2,.6}
\definecolor{orange}{rgb}{0.7,0.3,0}
\definecolor{blue}{rgb}{.2,.6,.75}
\definecolor{green}{rgb}{.4,.7,.4}

\newcommand{\suppress}[1]{}

\numberwithin{equation}{section}

\title{Ranks of elliptic curves over $\Q(T)$ of small degree in $T$} 

\author[F.~Battistoni]{Francesco Battistoni}
\address{Laboratoire de math\'ematiques de Besan\c con, Universit\'e de Bourgogne Franche-Comt\'e, CNRS UMR 6623, 16, route de ray, 25030 Besan\c con cedex, France}
\email{francesco.battistoni@univ-fcomte.fr}

\author[S.~Bettin]{Sandro Bettin}
\address{Dipartimento di Matematica, Università di Genova, Via Dodecaneso 35, 16146 Genova, Italy}
\email{bettin@dima.unige.it}

\author[C.~Delaunay]{Christophe Delaunay}
\address{Laboratoire de math\'ematiques de Besan\c con, Universit\'e de Bourgogne Franche-Comt\'e, CNRS UMR 6623, 16, route de ray, 25030 Besan\c con cedex, France}
\email{christophe.delaunay@univ-fcomte.fr}

\makeatletter
\@namedef{subjclassname@2020}{%
  \textup{2020} Mathematics Subject Classification}
\makeatother

\begin{document}

\begin{abstract}
We study elliptic surfaces over $\Q(T)$ with coefficients of a Weierstrass model being polynomials in $\Q[T]$ with degree at most 2. We derive an explicit expression for their rank over $\Q(T)$ depending on the factorization and other simple properties of certain polynomials.  Finally, we give sharp estimates for the ranks of the considered families and we present several applications, among which there are lists of rational points, generic families with maximal rank and generalizations of former results.
\end{abstract}

\subjclass[2020]{11G05, 14G05}

\keywords{Rational elliptic surface, rank, elliptic curves, rational points.}

\maketitle


\section{Introduction}
Consider an elliptic curve defined over the function field $\Q(T)$ and with Weierstrass model
\es{\label{eqine}
\E \colon Y^2= X^3 + \alpha_2(T)X^2 + \alpha_4(T) X + \alpha_6(T)
}
where $\alpha_2, \alpha_4$ and $\alpha_6$ are polynomials in $\Q[T]$. It is known (see~\cite{ellipticSurfaces}) that if $\deg \alpha_i\leq i$ for $i=2,4,6$, then $\E$ is a rational elliptic surface over $\Q$.  

In this paper we are interested in the rank $r=r_{\E}$ of $\E(\Q(T))$. Aside from its intrinsic geometric
interest for elliptic surfaces, the rank is also an arithmetical invariant that is related to several questions in number theory. For example, 
Silverman's specialization theorem asserts that for almost all specializations of $T$ at $t \in \Q$, the rank of the associated elliptic curve defined over $\Q$ 
is at least $r_{\mathcal E}$: this has direct consequences on the study of the distribution of ranks of families of elliptic curves defined over $\Q$, 
or over number fields $K$ \cite{miller, david_all, stewart_top, rubin_silverberg} and on the research of elliptic curves with high rank
\cite{mestre, fermigier, arms_all}.
The study of the 
rank, $r_{\mathcal E}$, has also 
some impact on a specific question in arithmetic geometry, concerning whether the set ${\mathcal E}(\Q)$ is Zariski-dense in ${\mathcal E}$:  
the question is positively answered whenever $r_{\mathcal E} >0$, and if  $r_{\mathcal E}=0$ a sufficient criterion consists in showing that there 
exist  infinitely many specializations of  $T$ at $t \in \Q$ such that the rank of the corresponding curve over $\Q$ is positive (see \cite{mazur} and 
recent works by J.~Desjardins on the 
subject, \cite{desjardins1, desjardins2}). For example, under the parity conjecture this can be done by studying the behavior of the root 
numbers of the specializations. 

In the works cited above, it is classical to consider $\E$ given as a polynomial in $X$ (and $Y$) with coefficient in $\Q(T)$; the rank of $\E$ over 
$\overline{\Q}(T)$ may then be recovered from the Shioda-Tate formula which is given for non-isotrivial rational elliptic surfaces by (\cite{ellipticSurfaces}):
\begin{equation}\label{rankclos}
r_{\E/\overline{\Q}(T)} = 8 - \sum_v (m_v-1),
\end{equation}
where $m_v$ is the number of the distinct irreducible components of the reduction of $\E$ at $v$. Whenever $\E$ runs through a specific family of 
rational elliptic surfaces, a good knowledge of the reductions (implying some restriction on $\alpha_2(T), \alpha_4(T)$ and $\alpha_6(T)$) allows one to deduce 
more arithmetical information such as, for example, the rank over $\Q(T)$ and a set of points (that have to be found in an ad hoc way) of $\E(\Q(T))$ generating 
a subgroup of finite index and finally to address arithmetical-geometry questions for the family.
\medskip

Our approach is instead different since the use of Nagao's formula (see Conjecture~\ref{nagao}) leads us to consider $\E$ rather as a 
polynomial in $T$ (and $Y$) with coefficient in $\Q[X]$ (see Equation~\eqref{familyCurve}). 
We then obtain a closed formula  for the rank $r_\E$ of $\E(\Q(T))$, involving simple properties of these polynomial coefficients, for all $\E$ such that the degrees of $\alpha_2, \alpha_4$ and $\alpha_6$ are $\leq 2$. Furthermore, the proof of the formula gives 
naturally a set of $r_\E$ points in $\E(\Q(T))$ which appear to be good candidates for generating a finite index subgroup. However, in general
we do not have enough control on the reduction type and on the geometry in order to deduce the independence of the points.
 
As mentioned above, we consider the case of $\deg(\alpha_i)\leq 2$ for $i=2,4,6$ in~\eqref{eqine} and find convenient to rewrite the Weierstrass model as
\begin{equation}\label{familyCurve}
\E \colon Y^2 = A(X) T^2 + B(X)T +C(X),    
\end{equation}
where $A,B,C\in\Q[X]$ with $\deg(A),\deg(B)\leq2$ and $C(X)$ a monic polynomial of degree $3$. We assume that at least one between $A(X)$ or $B(X)$ is not the zero polynomial, otherwise $\E$ is a constant elliptic curve over $\Q(T)$.

Before stating our main theorem, we need to fix some notation.
For a non-zero polynomial $P\in\Q[X]$ we let $\Omega(P)$ and $\omega(P)$ be the number of irreducible factors of $P$ counted with and without multiplicity respectively. Also, we let $P^*$ be the product of the (monic) irreducible factors dividing $P$ and we define $\square(P):=1$ if $P$ is a non-zero square in $\Q[X]$ and $\square(P):=0$ otherwise. Also, for $0\neq k\in\Q$ and $P$ irreducible we let $\sigma(k,P):=1$ if $k$ is a square in $\Q[X]/P(X)$ and $\sigma(k,P):=0$ otherwise. If $P$ is not irreducible, we let
\est{
\sigma(k,P):=\sum_{F|P,\atop F\text{ irr.}}\sigma(k,F).
}
In particular, $\sigma(k,P)=\omega(P)$ if $k\in\Q^2$.

Given two polynomials $P_1,P_2\in\Q[X]$ with $P_2\neq0$ we define $\gcd(P_1,P_2)\in\Q[X]$ to be the (monic) greatest common divisor of $P_1$ and $P_2$ and we let $M_{P_1,P_2}(X)\in\Q[X]$ to be the resultant (with respect to the variable $Y$) of $P_1(Y)$ and $X^2-P_2(Y)$, that is
\es{\label{dfrs}
M_{P_1,P_2}(X) := \Res_Y(P_1(Y),X^2-P_2(Y)).
}
Also, we let $\Upsilon_{P_1,P_2}$ be defined as in~\eqref{dfups}. Finally, if $A$ is a linear polynomial with a zero $\alpha$ we let 
\es{\label{dfxi}
\Xi_{A,P_1,P_2}:=\begin{cases}
1-\square(P_2(\alpha)) & \text{if $A|P_1$,}\\
0 & \text{ otherwise.}
\end{cases}
}

\begin{theo} \label{mtt}
Let $A,B,C\in\Q[X]$ with $\deg(A),\deg(B)\leq2$ and $C(X)$ a monic polynomial of degree $3$. Also, assume $A$ and $B$ are not both zero. Then, the rank over $\Q[T]$ of the elliptic curve $\E$ defined in~\eqref{familyCurve} is
\est{
r_\E=\begin{cases}
\Omega(M_{B^*,C})-2\deg(\gcd(B^*,C))+\omega(\gcd(B, C))-\omega(B)-\Upsilon_{B,C} & \text{if $A=0$,}\\
\sigma(A,B^2-4AC)-\square(k)  & \text{if $A\in\Q\setminus\{0\}$,}\\
\Omega(M_{(B^2-4AC)^*,A})-\omega(B^2-4AC)-\Xi_{A,B,C} & \text{if $\deg(A)=1$,}\\
\end{cases}
}
whereas for $\deg(A)=2$,
\est{
r&=\Omega(M_{\gcd(A,B^*),C})+\Omega(M_{(B^2-4AC)^*,A})+\omega(\gcd(A,B,C))-\Upsilon_{\gcd(A,B),C}-\Upsilon_{B^2-4AC,A}+{}\\
&\quad-\omega(B^2-4AC)-2\deg(\gcd(A,B^*,C)) -2\deg(\gcd(A,B^*))-\square(A).
}
\end{theo}

\begin{remark}\label{rmktwist} 
Theorem \ref{mtt} can also be used with quadratic twists of elliptic surfaces thus allowing one to recover all the results concerning the rank in the Propositions 8-12 in \cite{bettinDavidDelaunay}. Indeed, it is sufficient to observe that for $w\neq 0$, the twisted curve  $E_{w}: w Y^2 = A(X)T^2 + B(X)T + C(X)$ is 
isomorphic to 
$$
Y^2 = wA(X/w)T^2 + w^2B(X/w)T +w^3C(X/w)
$$
where $w^3C(X/w)$ is a monic polynomial of degree 3; we can then apply the formulae for the rank given in Theorem~\ref{mtt} to these polynomials. 
Notice that one can use the fact that $\Omega(M_{\lambda P_1,w P_2}) = \Omega(M_{P_1,\frac{1}{w}P_2})$ for any non-zero $\lambda, w \in \Z$ to simplify a little bit the final formula.
\end{remark}

Theorem~\ref{mtt} gives the rank $r_{\E}$ in terms of easily computable properties of the polynomials $A,B,C$ such as the factorization of $B^2-4AC$ or of $M_{(B^2-4AC)^*,A}$. The only quantity which might appear less trivial to determine is $\sigma(A,B^2-4AC)$; it consists in 
detecting whether or not an element $k$ is a square in a number field $K$; this is done by factorizing the polynomial $X^2- k$ in $K[X]$ which in return reduces to factorize some polynomials in $\Q[X]$, see \cite[Section 3.6]{cohen0} for details. 
Moreover, Theorem~\ref{mtt} indicates clearly the properties that generate a positive rank. Indeed, if $\deg A\in\{1,2\}$ or if $A\in\Q\setminus\Q^2$, then a positive contribution to the rank $r_{\E}$ typically corresponds to a (Galois conjugacy class of a) root $\rho$ of the polynomial $B^2-4AC$ such that $A(\rho)$ is a square in the number field $\Q(\rho)$. Similarly, if $A=0$, then positive contributions arise from the roots $\rho$ of $B$ such that $C(\rho)$ is a square in the number field $\Q(\rho)$. If instead $A\in\Q^2\setminus\{0\}$, then the rank is just one less than the number of distinct irreducible factors of the radical of $B^2-4AC$. As mentioned above, in all theses cases the Theorem and its proof also suggest natural candidates for independent points. We refer to Section~\ref{appl} for a more detailed discussion as well as more results. In all these cases we will also provide examples and construct families where the rank is maximal.

The structure of the paper is as follows. Section~\ref{sectionb} is devoted to the proof of Theorem~\ref{mtt}. More specifically, in Section~\ref{firstdecomp} we apply Nagao's formula to the curve $\E$ and we evaluate the resulting sums reducing the problem to that of studying averages of sums of the form $S_{P_1,P_2}(p):=\sum_{\substack{ P_1(x) \equiv 0 \mod p}} \leg{P_2(x)}{p}$, where $\leg{\cdot}\cdot$ is the Legendre symbol. We will then present the technical tools required in estimating these averages, with an important role played in particular by Chebotarev's theorem and by properties of the resultant. In Section~\ref{section3} we compute the average value of $S_{P_1,P_2}(p)$ in the relevant cases and in Section~\ref{section4} we use these results to prove Theorem~\ref{mtt}. Section~\ref{appl} is divided in several subsections, one for each possible degree of $A$. For every such subcase, we provide sharp estimates on the rank, show how to obtain generic families having maximal rank and give explicitly some points which we believe would typically generate a finite index subgroup. 
In Section~\ref{difftype} we show how the same method can be applied to the case where $A$ is the cube of a linear polynomial and $\deg(B),\deg(C)\leq3$, recovering and generalizing results from~\cite{arms_all}. Finally, Section~\ref{algrem} deals with the computational features which we took into account in order to provide examples throughout the paper; these arise as results of computations done on the computer algebra packages PARI/GP \cite{pari} and Magma \cite{cannon2011handbook}. We also suggest a possible application to the research of rational points for elliptic curves defined over $\Q$.

\subsection*{Acknowledgments} The first and the last author are supported  by the French ``Investissements d'Avenir" program, project ISITE-BFC (contract ANR-lS-IDEX-OOOB). The second author is member of the INdAM group GNAMPA and his work is partially supported by PRIN 2017 ``Geometric, algebraic and analytic methods in arithmetic''. We thank Chantal David, and Nicolas Mascot for a useful suggestion (see Section~\ref{mascot}).

During the final preparation of this article, we were informed by M.~Sadek that he also used a strategy based on Nagao's conjecture in his upcoming work~\cite{Sadek} where he obtained upper and lower bounds for $r_{\E}$ depending on the factorization of $B^2-4AC$.

\section{Proof of Theorem~\ref{mtt}}\label{sectionb}

\subsection{A first decomposition}\label{firstdecomp}
As in~\cite{arms_all}, our starting point is Nagao's conjecture~\cite{rosen-silverman}.

\begin{conj}[Nagao]\label{nagao} 
Let $\E$ as in~\eqref{eqine}. Then, the rank $r$ of $\E(\Q(T))$ satisfies
$$
r = \lim_{Q\rightarrow \infty} \frac{\log Q}{Q} \sum_{p \leq Q} -A_{\E}({p})
$$
where the sum runs through all prime numbers $p \leq Q$ and
$$
A_{\E}({p}) = \frac{1}{p}\sum_{t=0}^{p-1} a_t({p})
$$
with $p+1 - a_t({p})$ being the number of points of $\E_t(\F_p)$ unless $p$ divides the discriminant of $\E_t$, and in this case $a_t({p}) =0$ (where $\E_t$ denotes the elliptic curve defined over $\Q$ obtained 
by specialization of $T$ at $t\in\Q$). 
\end{conj}
The conjecture was proven by Rosen and Silverman \cite{rosen-silverman} for rational elliptic surfaces. In particular, since the hypothesis $\deg(\alpha_i)\leq2$ implies that $\E$ is in fact a rational elliptic surface (\cite{ellipticSurfaces}), then Nagao's conjecture is actually a theorem in the case we are considering.

Expressing $a_t({p})$ in terms of the Legendre symbol $(\frac \cdot\cdot)$, we can write $-A_{\E}(p)$ as 
\begin{align*}
-A_{\E}(p)&=\frac1p\sum_{x,t \mod{p}} \left(\frac{A(x) t^2 + B(x)t +C(x)}{p}\right)
\end{align*}
provided that $p$ doesn't divide the discriminant of $\E$. Summing over $t$, the contribution of the $x\mod p$ such that $A(x)\equiv 0$ is easily seen to be $(\frac{C(x)}{p})$ if also $B(x)\equiv 0$ and it is zero otherwise. If $A(x)\nequiv 0\mod p$ then the sum over $t$ is a complete quadratic Legendre sum which can be evaluated exactly (see e.g.~\cite[Theorem 5.48]{finite_fields}). These two cases together give 
\begin{align*}
-A_{\E}(p) &=\sum_{\substack{x \mod p \\ A(x)\equiv B(x) \equiv 0 \mod p}} \leg{C(x)}{p}+ \sum_{\substack{x \mod p \\ (B^2-4AC)(x) \equiv 0 \mod
p}} \leg{A(x)}{p} - \frac{1}{p} \sum_{\substack{x \mod p }} \leg{A(x)}{p}.
\end{align*}
By Weil's bound the last sum is $-1$ if the polynomial $A(X)$ 
is a non-zero square in $\Q[X]$ and it is $O(p^{-1/2})$ otherwise. 
Thus, defining\footnote{If $P_1$ or $P_2$ are not defined mod $p$ we simply let $S_{P_1,P_2}(p):=0$. The actual choice is irrelevant since we are concerned only with the case of large $p$.}
\est{
S_{P_1,P_2}(p):=\sum_{\substack{x \mod p \\ P_1(x) \equiv 0 \mod p}} \leg{P_2(x)}{p}
} 
for two polynomials $P_1(X),P_2(X)\in\Q[X]$ and a sufficiently large prime $p$, we obtain 
\es{\label{msfr}
r=\lim_{Q\to\infty}\frac{\log Q}Q\sum_{p\leq Q}(S_{\gcd(A,B),C}(p)+S_{B^2-4AC,A}(p))-\square(A)
}
(recall that by definition $\square(0)=0$) which will be the starting point of our analysis. In the following subsections we shall determine the average behaviour of $S_{P_1,P_2}(p)$ in the relevant cases. Mostly, our computations will be based upon consequences of Chebotarev's density theorem, as given in the following two lemmas.

\begin{defin}\label{dfa2}
Let $k\in\N$, $p$ a prime and let $P\in \Q[X]\setminus\{0\}$. Then, if $p$ divides $P$ or a denominator of the coefficients of $P$ we let $N_{P,k}(p):=0$, otherwise we define $N_{P,k}(p)$ to be the number of degree $k$ irreducible factors of $P(X)$ in $\F_p[X]$ counted with multiplicity. Notice that $N_P(p):=N_{P,1}(p)$ is the number of zeros  (counted with multiplicity) of $P(X)$ modulo $p$.
\end{defin}
\begin{lemma}\label{omega}
For a non-zero polynomial $P\in\Q[X]$ we have
\begin{equation}\label{limitomega}
\lim_{Q \rightarrow \infty} \frac{\log Q}{Q} \sum_{p \leq Q} N_P(p) = \Omega(P).
\end{equation}
\end{lemma}
\begin{proof}
It is enough to show this for $P$ irreducible, since $N_{P_1\cdot P_2}(p) = N_{P_1}(p)+N_{P_2}(p)$ for $P_1,P_2\in \Q[X]$. If $P$ is irreducible, let $K:= \Q[X]/P(X)$ be its associated number field and let $\mathcal{O}_K\subset K$ be the corresponding ring of integers: then 
$$
\lim_{Q \rightarrow \infty} \frac{\log Q}{Q} \sum_{p \leq Q} N_P(p) = 
\lim_{Q \rightarrow \infty} \frac{\log Q}{Q}\sum_{\substack{\mathfrak{p}\subset\mathcal{O}_K\\\mathfrak{p}\text{ prime, }\mathcal{N}(\mathfrak{p})\leq Q\\ \mathfrak{p}\text{ has inertia degree 1}}}1
$$
where $\mathcal{N}(\mathfrak{p})$ is the absolute norm of $\mathfrak{p}$. Then the claim follows from the Prime Ideal Theorem, see e.g.~\cite[Chapter VIII.2]{cassels1967algebraic}.  
\end{proof}

\begin{lemma}\label{densityt} 
Let $k\in\Q$ and let $P(X)\in\Q[X]$ be an irreducible polynomial. Then, 
\begin{equation}\label{limitsigma}
    \lim_{Q\rightarrow\infty} \frac{\log Q}{Q}\sum_{p\leq Q} N_P(p)\leg{k}{p} = \sigma(k,P).
\end{equation}
\end{lemma}
\begin{proof}
Let $K:= \Q[X]/P(X)$:  notice that $k$ is a square in $K$ if and only if $\Q(\sqrt{k})\subseteq K$.
\\
Assume $\Q(\sqrt{k})\subseteq K$. Up to finitely many exceptions, for a prime number $p$ we have $N_P(p)=r>0$ if and only if there exist $r$ primes $\mathfrak{p}\subset\mathcal{O}_K$ dividing $p$ with inertia 1 over $\Q$: by multiplicativity of inertia degrees, every prime $p$ such that $N_P(p)>0$ splits in $\Q(\sqrt{k})$, i.e. $\left(\frac{k}{p}\right) = 1.$ Thus,
\begin{align*}
 \lim_{Q\rightarrow\infty} \frac{\log Q}{Q}\sum_{p\leq Q} N_P(p)\leg{k}{p}=
     \lim_{Q\rightarrow\infty} \frac{\log Q}{Q}\sum_{p\leq Q} N_P(p) = 1
\end{align*}
thanks to \eqref{limitomega}.

Assume now $K\cap \Q(\sqrt{k})=\Q$: the composite field is $K(\sqrt{k}) \simeq K[X]/(X^2-k)$, which is a quadratic (hence Galois) extension of $K$. Consider
\begin{equation}\label{limitintermediate}
    \lim_{Q\to\infty}\frac{\log Q}{Q}\sum_{\substack{p\leq Q\\\left(\frac{k}{p}\right)=1}}N_P(p)\left(\frac{k}{p}\right) = 
    \lim_{Q\to\infty}\frac{\log Q}{Q}\sum_{\substack{p\leq Q\\\left(\frac{k}{p}\right)=1}}N_P(p).
\end{equation}
If we are able to prove that this limit is equal to 1/2, then, since
\begin{align*}
    1 = \lim_{Q\to\infty}\frac{\log Q}{Q}\sum_{\substack{p\leq Q}}N_P(p) = 
    \lim_{Q\to\infty}\frac{\log Q}{Q}\sum_{\substack{p\leq Q\\\left(\frac{k}{p}\right)=1}}N_P(p) + \lim_{Q\to\infty}\frac{\log Q}{Q}\sum_{\substack{p\leq Q\\\left(\frac{k}{p}\right)=-1}}N_P(p)
\end{align*}
we obtain that also the third limit is equal to $1/2$, implying that the left hand side of \eqref{limitsigma} is equal to 0. 
The left hand side of \eqref{limitintermediate} may be rewritten as
\begin{equation*}
    \lim_{Q\to\infty}\frac{\log Q}{Q}\sum_{\substack{\mathfrak{p}\subset\mathcal{O}_K\\\mathfrak{p}\text{ has inertia }1\\\mathcal{N}(\mathfrak{p})\leq Q, \left(\frac{k}{\mathfrak{p}\cap\Z}\right)=1}}1.
\end{equation*}
But now, since $\mathfrak{p}$ has inertia 1 over $\Q$, we have $\mathcal{O}_K/\mathfrak{p} \simeq \mathbb{F}_{\mathfrak{p}\cap\Z}$ and in particular the polynomial $X^2-k$ splits in $(\mathcal{O}_K/\mathfrak{p})[X]$. This means that, up to finitely many exceptions, $\mathfrak{p}$ splits in the quadratic extension $K(\sqrt{k})/K$ and thus
\begin{equation*}
    \lim_{Q\to\infty}\frac{\log Q}{Q}\sum_{\substack{\mathfrak{p}\subset\mathcal{O}_K\\\mathfrak{p}\text{ has inertia }1\\\mathcal{N}(\mathfrak{p})\leq Q, \left(\frac{k}{\mathfrak{p}\cap\Z}\right)=1}}1=
    \lim_{Q\to\infty}\frac{\log Q}{Q}\sum_{\substack{\mathfrak{p}\subset\mathcal{O}_K\\\mathfrak{p}\text{ has inertia }1\\\mathcal{N}(\mathfrak{p})\leq Q\\ \mathfrak{p}\text{ splits in }K(\sqrt{k})}}1.
\end{equation*}
By the Prime Ideal Theorem this limit is equal to
$$
\lim_{Q\to\infty}\frac{\log Q}{Q}\sum_{\substack{\mathfrak{p}\subset\mathcal{O}_k\\\mathcal{N}(\mathfrak{p})\leq Q\\ \mathfrak{p}\text{ splits in }K(\sqrt{k})}}1
$$
and the result  is $1/2$ by Chebotarev's Theorem.
\end{proof}

\subsection{The average value of $S_{P_1,P_2}(p)$}\label{section3}
We first record some easy properties of $S$ and some first easy cases following from the definition of $S$ or from Lemma~\ref{omega} and Lemma~\ref{densityt}.
\begin{prop}\label{rp1}
Let $P_1,P_2\in\Q[X]$, with $P_1\neq0$. Then for $p$ large enough we have
\est{
S_{P_1,P_2}(p)=\sum_{F|P_1,\atop F\text{irr.}}S_{F,P_2}.
}
\end{prop}
\begin{proof}
This is immediate from the definition since coprime polynomials have different roots mod $p$ if $p$ is sufficiently large.
\end{proof}

\begin{prop}\label{rp2}
Let $P_1,P_2\in\Q[X]$, with $P_1\neq0$. Then if $P_1|P_2$ one has $S_{P_1,P_2}(p)=0$. Moreover, if $P_1$ is a linear polynomial with root $x_1$, then 
\est{
\lim_{Q\to\infty}\frac{\log Q}Q\sum_{p\leq Q} S_{P_1,P_2}(p)=\square(P_2(x_1))
}
\end{prop}
\begin{proof}
The first claim is clear. The second follows by the prime number theorem in arithmetic progressions since by definition $ S_{P_1,P_2}(p)=(\frac {P_2(x_1)}p)$.
\end{proof}
\begin{prop}\label{rp3}
Let $P_1\in\Q[X]$ and let $k\in\Q$. Then
\est{
\lim_{Q\to\infty}\frac{\log Q}Q\sum_{p\leq Q} S_{P,k}(p)=
\begin{cases}
0 & \text{if $k=0$},\\
\omega(P_1) & \text{if $k\in\Q^2\setminus\{0\}$},\\
\sigma(k,P_1)& \text{if $k\in\Q\setminus\Q^2$}.\\
\end{cases}
}

\end{prop}
\begin{proof}
This follows from Proposition~\ref{rp1}, Lemma~\ref{omega} and Lemma~\ref{densityt}.
\end{proof}

If $\deg(P_2)\geq1$ we need a more elaborate computation, which will lead to a result depending on the resultant $M_{P_1,P_2}$ (defined in~\eqref{dfrs}). We let 
$\delta_x:=1$ if $x=0$ and $\delta_x:=0$ otherwise and we define $\Upsilon_{P_1,P_2}$ as follows.

\begin{defin}
Let $P_1,P_2\in\Q[X]$ with $P_1$ irreducible and of degree less than or equal to $5$ and $P_2$ non constant and with leading coefficient $s$. Also let 
 $\min(\deg(P_1),\deg(P_2))\leq2$. Let $\D_{P_1}$ and $\D_{P_2}$ be the discriminant of ${P_1}$ and $P_2$ respectively.
Then, we define
\es{\label{dfups}
\Upsilon_{P_1,P_2}:=\begin{cases}
\delta_{t}(1-\delta_u+\square(u)-\square(uD_{P_1}))& \parbox[t]{.55\textwidth}{if $\deg(P_1)=2$ and $P_2(X)=Q(X)P_1(X) + (tX+u)$ with $t,u\in\Q$, $0\neq Q\in\Q[X]$;}\medskip\\
\parbox[c]{.3\textwidth}{ $\Upsilon_{U,W}+2-\Omega(M_{U,W})+{ + }{}\Omega(U(x^2))-\Omega(U(\tfrac{x^2-\D_{P_2}}{4s}))$}
& \parbox[c]{.55\textwidth}{if $\deg(P_1)=4$, $\deg(P_2)=2$ and \\$P_1(X)=U(P_2(X))$ for some $U\in\Q[X]$ and where\\ $W(X):=4sX^2+\D_{P_2} X$;}\medskip\\
0 & \text{otherwise.}
\end{cases}
}
If $P$ is not irreducible, but each of its irreducible factor $F$ satisfies the above conditions we define
\est{
\Upsilon_{P,P_2}=\sum_{F|P,\atop F\text{ irr.}}\Upsilon_{F,P_2}.
}
\end{defin}

\begin{prop}\label{rp4}
Let $P_1,P_2\in\Q[X]\setminus\{0\}$ with $P_2$ non constant and with $\min(\deg(F),\deg(P_2))\leq2$ and $\deg(F)\leq 5$ for each irreducible factor $F$ of $P_1$. Then,
\est{
\lim_{Q\to\infty}\frac{\log Q}Q\sum_{p\leq Q} S_{P_1,P_2}(p)&=\sum_{F|P_1,\atop F\text{ irr.}}\big(\Omega(M_{F,P_2})+\Omega(\gcd(F,P_2))-2\deg(\gcd(F,P_2))-\Upsilon_{F,P_2}-1\big)\\
&=\Omega(M_{P_1^*,P_2})+\omega(\gcd(P_1,P_2))-2\deg(\gcd(P_1^*,P_2))-\Upsilon_{P_1,P_2}-\omega(P_1),
}
where we recall that $P^*$ denotes the product of the irreducible factors dividing $P$.
\end{prop}
\begin{proof}
By Proposition~\ref{rp1} we can assume $P_1$ is irreducible. Also, we assume $p$ is sufficiently large. We start by introducing an extra sum,
\est{
S_{P_1,P_2}(p)&=\sum_{\substack{x \mod p \\ P_1(x) \equiv 0 \mod p}}\bigg(\sum_{\substack{\ell\mod p\\ P_2(x)\equiv \ell^2}}-1\bigg)
=\sum_{\ell=1}^{p-1}N_{\gcd(P_1,P_2-\ell^2)}(p)+N_{\gcd(P_1,P_2)}(p)-N_{P_1}(p).
}
By Lemma \ref{lemmaMult}, presented in the following, we obtain
\begin{align*}
S_{P_1,P_2}(p)&=N_{M_{P_1,P_2}}(p)-2\deg(\gcd(P_1,P_2))-2\sum_{\ell=0}^{p-1} N_{\gcd(P_1,P_2-\ell^2),2}(p)+N_{\gcd(P_1,P_2)}(p)-N_{P_1}(p)
\end{align*}
since $\deg(P_1,P_2-\ell^2)\leq2$, where $N_{P,2}$ is as in Definition~\ref{dfa2}. The case $\min(\deg(P_1,P_2))\leq1$ then follows at once by Lemma~\ref{omega} since trivially one has $N_{\gcd(P_1,P_2-\ell^2),2}(p)=0$. 

Now, let's assume $\deg(P_1)=2\geq\deg(P_2)$. We write  $P_2(X)=Q(X)P_1(X) + (tX+u)$ with $t,u\in\Q$, $0\neq Q\in\Q[X]$. 
Clearly, if $t\neq0$ then $\deg(\gcd(P_1,P_2-\ell^2))\leq1$ for  $p$ large enough and thus $N_{\gcd(P_1,P_2-\ell^2),2}(p)=0$. If instead $t=0$ then $\gcd(P_1,P_2-\ell^2)=P_1$ if $u\equiv\ell^2\mod p$ (which happens for two values of $\ell$ if $u$ is a non-zero square modulo $p$) and $\gcd(P_1,P_2-\ell^2)=1 $ otherwise. Since $P_1$ is reducible modulo $p$ if and only if $\D_{P_1}$ is a square modulo $p$ we then obtain in this case
\est{ 
\sum_{\ell\nequiv 0\mod p}2N_{\gcd(P_1,P_2-\ell^2),2}(p)=\frac12\bigg(1-\leg{\D_{P_1}}{p}\bigg)\bigg(\leg{u}{p}+\leg{u^2}{p}\bigg)
}
for $p$ large enough. Thus, recalling that $\D_{P_1}$ is not a square in $\Q$ since $P_1$ is irreducible, the claimed equality follows from the prime number theorem in arithmetic progressions.

Finally assume $3\leq \deg(P_1)\leq 5$ and $\deg(P_2)=2$. We write $P_1(X)=U(P_2(X))+XV(P_2(X))$ for some $U,V\in\Q[X]$ of degree $\leq2$ (which is easily seen to be always possible by the Euclidian division). Now, we claim that if $P_2-\ell^2 \mid P_1$ in $\F_p[X]$ for $p$ large enough then $V=0$, so that for $V\neq0$ one has $N_{\gcd(P_1,P_2-\ell^2),2}(p)=0$ and the result follows in this case. To see this, we notice that $P_2-\ell^2 \mid P_1$ if and only if $U(\ell^2) + \theta V(\ell^2) =0$ in $\F_p[\theta]$ where $\theta$ is the image of $X$ in the $2$-dimensional $\F_p$-vector space $\F_p[X]/(P_1(X)-\ell^2)$ of dimension 2. Since $(1,\theta)$ is a basis of this vector space, we must have $U(\ell^2) = V(\ell^2) =0 \bmod p$ and hence $U$ and $V$ have a common root in $\F_p$. This can happen for infinitely many primes modulo $p$ only if $U,V$ have a common factor in $\Q[X]$ and thus, since $P_1$ is irreducible, we must have $V=0$, as claimed. 

We are thus left with the case where $P_1(X)=U(P_2(X))$ for some (necessarily irreducible) degree $2$ polynomial $U$. As before we have $(P_2-\ell^2)|P_1$ if and only if $U(\ell^2)\equiv 0$ and $P_2-\ell^2$ is irreducible if and only if $\D_{P_2-\ell^2}=\D_{P_2}+4s\ell^2\equiv 0\mod p$. Thus, 
\est{ 
\sum_{\ell=0}^{p-1}2N_{\gcd(P_1,P_2-\ell^2),2}(p)&=\sum_{\substack{\ell'\nequiv 0\mod p,\\ U(\ell')\equiv0}}\bigg(1-\leg{\D_{P_2}+4s{\ell'}}{p}\bigg)\bigg(1+\leg{\ell'}{p}\bigg)\\
&=N_{U}(p)+S_{U,X}-S_{U,4s{X}+\D_{P_2}}-S_{U,4sX^2+\D_{P_2}\!X}
}
for $p$ large enough. The result then follows by appealing to the proposition in the case $\deg(P_1)=2$,
 since $M_{U,X}(X)=U(X^2)$, $M_{U,4s{X}+\D_{P_2}}(X)=(4s)^2U(\frac{X^2-\D_{P_2}}{4s})$. 

\end{proof}

\begin{lemma}\label{lemmaMult}
Let  $P_1(X),P_2(X)\in\Q[X]$ with $P_1$ square-free and let $p$ be large enough. Then, for $0\nequiv \ell\mod p$ we have
$$
m_{M_{P_1,P_2}}(\ell)= \sum_{k\geq 1} k\cdot N_{\gcd(P_1,P_2-\ell^2),k}(p),
$$
where $m_{M_{P_1,P_2}}(\ell)$ is the multiplicity of $\ell$ as a zero of $M_{P_1,P_2}$ in $\F_p[X]$. If $\ell\equiv 0\mod p$ then $m_{M_{P_1,P_2}}(0)$ is equal 
to $2\deg(\gcd(P_1,P_2))$.
\end{lemma} 
\begin{proof}
Let $p$ be large enough so that it doesn't divide the discriminant and the leading coefficient of $P_1$ (as well as the denominators of $P_1$ and $P_2$). Then, since $\mu-\ell^2$ has two distinct roots whenever $\mu\nequiv0\mod p$, by the standard properties of the resultant for $\ell\equiv 0$ we have
\est{
m_{M_{P_1,P_2}}(\ell)=\sum_{\substack{\rho\in\overline\F_{p}[X],\\ P_1(\rho)\equiv 0,\ P_2(\rho)-\ell^2 \equiv 0
}}1=\sum_{k\geq1}k\cdot N_{\gcd(P_1,P_2-\ell^2),k}(p),
}
as desired. For $\ell\equiv 0$ we have that each zero $\rho$ gives a zero of multiplicity two of $M_{P_1,P_2}(X)$ at $X=0$ and the result follows.
\end{proof}

\subsection{Proof of Theorem~\ref{mtt}}\label{section4}
We are now ready for the proof of Theorem~\ref{mtt}. In order to do so, we start by applying~\eqref{msfr} and then one of the above propositions depending on the case considered. We recall that we are assuming that $A$ and $B$ are not both identically zero. 

If $A=0$ and $B$ is constant, then the definition of $S$ give trivially $r=0$. 
If $A=0$ and $B$ is non constant, then Propositions~\ref{rp2} and~\ref{rp4} give
\est{
r=\lim_{Q\to\infty}\frac{\log Q}Q\sum_{p\leq Q}S_{B,C}(p)=\Omega(M_{B^*,C})+\omega(\gcd(B, C))-2\deg(\gcd(B^*,C))-\omega(B)-\Upsilon_{B,C}
}

If $A=k\in\Q\setminus\{0\}$, then Proposition~\ref{rp3} gives
\est{
r=\lim_{Q\to\infty}\frac{\log Q}Q\sum_{p\leq Q}S_{B^2-4kC,k}(p)-\square(k)=\begin{cases}
\omega(B^2-4kC)-1 & \text{if $k\in\Q^2\setminus\{0\}$},\\
\sigma(k,B^2-4kC)& \text{if $k\in\Q\setminus\Q^2$}.\\
\end{cases}
}

If $\deg(A)=1$ with $A|B$, then Proposition~\ref{rp1},~\ref{rp2} and~\ref{rp4} give
\est{
r&=\lim_{Q\to\infty}\frac{\log Q}Q\sum_{p\leq Q}(S_{A,C}(p)+S_{B^2-4AC,A}(p))\\
&=\square(C(\alpha))+\Omega(M_{(B^2-4AC)^*,A})-\omega(B^2-4AC)-1,
}
where $\alpha$ is the root of $A$. If $A\nmid B$, then $\gcd(A,B)=\gcd(B^2-4AC,A)=1$ and the same computation gives
\est{
r&=\Omega(M_{(B^2-4AC)^*,A})-\omega(B^2-4AC).
}

Finally, if $\deg(A)=2$, by Proposition~\ref{rp4} we have
\est{
r&=\lim_{Q\to\infty}\frac{\log Q}Q\sum_{p\leq Q}(S_{\gcd(A,B),C}(p)+S_{B^2-4AC,A}(p))-\square(A)\\
&=\Omega(M_{\gcd(A,B^*),C})+\omega(\gcd(\gcd(A,B),C))-2\deg(\gcd(\gcd(A,B^*),C))-\Upsilon_{\gcd(A,B),C}+{}\\
&\quad-\omega(\gcd(A,B))+\Omega(M_{(B^2-4AC)^*,A})+\omega(\gcd(B^2-4AC,A))-2\deg(\gcd((B^2-4AC)^*,A))+{}\\
&\quad-\Upsilon_{B^2-4AC,A}-\omega(B^2-4AC)-\square(A)\\
}
and the result follows since $\gcd((B^2-4AC)^*,A) = \gcd(A,B^*)$ and $\gcd(C,\gcd(A,B)^*) = \gcd(A,B^*,C)$.

\section{Applications: estimates for the rank, rational points, generic families}\label{appl}
The goal of this section is to provide an effective explanation for the maximum value of $r = \rank(\E/\Q(T))$ of the family $\E$ defined in \eqref{familyCurve} for each possible degree of $A$ (i.e. $\deg A=0, 1$ or $2)$. Furthermore, we provide explicit rational points in every case such that $r\geq 1$, and we present applications of our results.
~\\
 We first state a general lemma about the irreducible factors of $M_{F,G}(X)$. We recall that $M_{F,G}$ is defined in~\eqref{dfrs} and that, for an irreducible  monic $F$, it can be written as
\est{
M_{F,G}(X)=\prod_{\rho \colon F(\rho)=0} (X^2-G(\rho)).
}

\begin{lemma}\label{factors}
Let $F(X), G(X)\in\Q[X]$ with $G$ of degree $1$ or $2$ and $F$ monic irreducible of degree $d$. Then 
\begin{itemize}
\item If one (and thus all) zero $\rho$ of $F$ is such that $G(\rho)$ is a square in 
$\Q(\rho)$ then $M_{F,G}(X) = (-1)^d K(X)^eK(-X)^e$ where $K(X) \in \Q[X]$ is irreducible and where $e=1$ or $2$.  Furthermore, $e=2$ if and 
only if $F(X)=P(G(X))$ for a certain polynomial $P$ of degree $d/2$ (in particular $P$ is irreducible and $G$ has degree 2).  
\item If $G(\rho)$ is not a square in 
$\Q(\rho)$ for one (and thus any) zero $\rho$ of $F$ then $M_{F,G}(X)=K(X^2)^e$ with $K \in \Q[X]$ such that $K(X^2)$ is irreducible and where $e=1$ or $2$. Furthermore, $e=2$ if 
and only if   $F(X)=P(G(X))$ for a certain polynomial $P$ of degree $d/2$ (in particular $P$ is irreducible and $G$ has degree 2).
\end{itemize}
\end{lemma}
\begin{proof}
Let $Z= \{ \rho \colon F(\rho)=0\}$. First, let's assume that for some $\rho\in S$ we have that $G(\rho)$ is a square in $\Q(\rho)$, i.e. that there exists a polynomial $R\in\Q[X]$ (of degree $\leq2$) such that $R(\rho)^2=G(\rho)$. Clearly we have $M_{F,G}(X)=(-1)^dK_1(X)K_1(-X)$ for $K_1(x)=\prod_{\rho \in Z} (X-R(\rho))$. Moreover, by Galois theory $K_1(X)\in\Q[X]$. Also, again by Galois theory, any irreducible factor of $K_1$ has to vanish at $R(\rho)$ for all $\rho\in S$. It follows that $K_1(X)=K(X)^e$ for some irreducible $K(X)$ and some $e\geq1$. Notice that for each  $\rho \in Z$ there are exactly $e$ zeros in $Z$ such that $R$ takes the value $R(\rho)$. All such zeros are then also zeros of $G(x)-R(\rho)^2$ so that, since $\deg(G)\leq 2$, one must have $e\leq2$. 
Finally, let $e=2$ (and thus necessarily $\deg(G)=2$). Then, for some subset $Z'\subseteq Z$ we have $K(X)=\prod_{\rho\in Z'}(X-R(\rho))$ and thus $P(X)=\prod_{\rho\in Z'}(X-G(\rho))\in\Q[X]$. Since $P(G(X))$ is monic and has the same zeros of $F$  we must have $P(G(X))=F(X)$, as claimed. The vice versa is also clear.

Next, we assume that $G(\rho)$ is not a square in $\Q(\rho)$ for any $\rho\in Z$. Any irreducible factor $H(X)$ of $M_{F,G}(X)$ must have the shape $H(X)=\prod_{\rho \in Z_1} (X^2-G(\rho))$ for some subset $Z'\subseteq Z$ since by Galois theory we have that either both roots or no root of the irreducible polynomial $X^2-G(\rho)$ can be a zero of $H(X)$. Then, the same argument as above implies that $M_{F,G}(X)=K(X^2)^e$ for some irreducible $K(X^2)$ and some $e\in\{1,2\}$ and that  $e=2$ if and only if $F(X)=P(G(X))$ for some polynomial of degree $d/2$.
\end{proof}

\subsection{Estimates and applications when $A$ is a constant polynomial}

\subsubsection{The case $A=0$}
Let us assume first that $A=0$, which is equivalent to $\alpha_2, \alpha_4$ and $\alpha_6$ in~\eqref{eqine} having degree $\leq 1$. Then Theorem~\ref{mtt} gives 
\begin{equation}\label{fsds}
r =\Omega(M_{B^*,C})-2\deg(\gcd(B^*,C))+\omega(\gcd(B, C))-\omega(B)-\Upsilon_{B,C}.
\end{equation}

If $\theta\in\overline\Q$ is an algebraic number, we denote by $[\theta]_G$ its orbit under the action of $G=\Gal(\overline{\Q}/\Q)$. We remark that if $C\in \Q[X]$ is a polynomial, then the condition ``$C(\theta)$ is a square in $\Q(\theta)$" is independent of the representative chosen in 
$[\theta]_G$.

\begin{theo}\label{theoconjclass} Assume that $A=0$\suppress{and that $B^*$ does not divide $C$}. Then 
$$
r = \sharp \{ [\theta]_G \colon B(\theta)=0, C(\theta) \text{ is a non-zero square in } \Q(\theta) \}. 
$$
In particular $r\leq 2$. Furthermore, every value between 0 and 2 can be obtained by $r$.
\end{theo}
\begin{proof}
First, we remark that we can assume that $B^*$ does not divide $C$, because otherwise we have $M_{B^*,C} = X^{2\deg(B^*)}$, $\gcd(B^*,C) = B^*$ and $\Upsilon_{B,C}=0$. In particular, $r=0$, as claimed. 
\medskip

Now, assume $B$ is an irreducible polynomial of degree 2, so that in particular $B=\alpha B^*$ for some $\alpha\in\Q\setminus\{0\}$ (and thus $\gcd(B,C)=1$). We write $C=QB + tX + u$. If $t=0$ (and thus $u\neq0$) we have $M_{B^*,C}=(X^2-u)^2$ and so $\Omega(M_{B^*,C})=2(1+\square(u))$. Moreover, since $\Upsilon_{B,C}=1+\square(u)-\square(uD_{P_1})$, where $D_B$ is the discriminant of $B$, we have $r=\square(u)+\square(uD_{B})$. Thus, $r=0$ unless either $u$ or $uD_{B}$ is a square, in which case $r=1$. The Theorem then follows in this case, since one among $u$ and $uD_{B}$ is a square if and only if $C(\theta)=u$ is a non-zero square in $\Q(\theta)$ (where $\theta$ is any root of $B$).
If instead, $t\neq 0$, then $\Upsilon_{B,C}=0$ and by Lemma~\ref{factors} we obtain $\Omega(M_{B^*,C})= 2$ or $1$ depending on whether $C(\theta)=t\theta+u$ is a square or not in $\Q(\theta)$, as claimed.
\medskip

Assume next that $B^*$ is a degree 2 polynomial that splits in $\Q$, i.e. $B(X)=\alpha B^*(X)= \alpha(X-b_1)(X-b_2)$ with $\alpha\in\Q\setminus\{0\}$, $b_1,b_2\in\Q$, $b_1\neq b_2$. In this case  $\Upsilon_{B,C}=0$.  If $\gcd(B,C)=1$ (and thus $C(b_1),C(b_2)\neq0$) then $M_{B^*,C}=(X^2-C(b_1))(X^2-C(b_2))$. In particular, $\Omega(M_{B^*,C})=2+\square(C(b_1))+\square(C(b_2))$ and so $r=\square(C(b_1))+\square(C(b_2))$, as claimed. Moreover, it is clear that every value of $r$ in $\{0,1,2\}$ can be achieved.
If instead $\gcd(C,B)=(X-b_1)$, then $M_{B^*,C}=X^2(X^2-C(b_2))$. Thus, $\Omega(M(C,B)) =3+\square(C(b_2))$ and $r=\square(C(b_2))$, as desired.
\medskip

Finally, if $B^*= (X-b)$ is a degree one polynomial, then $\Upsilon_{B,C}=0$ and $M_{B^*,C}(X)=X^2-C(b)$. This implies $r=\square(C(b))$, which concludes the proof.
\end{proof}

\noindent

\begin{rem*}
The result of Theorem~\ref{theoconjclass} suggests that in the family of curves \eqref{familyCurve} with $A=0$ the generic curve has rank 0, since for every quadratic number field $\Q(\rho)$ the amount of cubic integer polynomials $C$ such that $C(\rho)$ is a square is negligible compared to the overall number of cubic polynomials. This remark fits with a conjecture by Cowan \cite{Cowan} stating that a typical elliptic surface has rank 0. We investigate this question in some cases in~\cite{bbd2}. 
\end{rem*}

As it is clear from the theorem, families of rank $2$ with $A=0$ can arise only if $B$ splits as a product of two distinct factors, so that up to a linear change 
of variables in $X$, we can assume that $B(X)=bX(X-1)$ for some $b\in\Q\setminus\{0\}.$ Writing $C(X)=X^3+c_2X^2+c_1X+c_0$ we thus have that $r=2$ 
if and only if $C(0)=c_0$ and $C(1)=1+c_2+c_1+c_0$ are both non-zero squares. If this is the case, writing $c_0=k^2$ and $1+c_2+c_1+c_0=\ell^2$, 
with $k\neq 0$ and $\ell\neq 0$,  we have that the rank $2$ curves of this form can be always reduced to
$$
\E\colon Y^2 = X^3 + (bT+c_2)X^2 + (-bT + \ell^2 -k^2 -c_2 - 1)X+k^2
$$
which could be simplified further with the linear change $T\leftrightarrow bT+c_2$ giving
\begin{equation}\label{example_geo}
Y^2 = X^3 + TX^2 -(T + k^2+1-\ell^2)X + k^2.
\end{equation}
Notice that we can also explicitly determine two points $P_1=(0,k)$ and $P_2=(1,\ell)$ in $\E(\Q(T))$ which are both easily verified to be non-torsion points.
Notice also that once the Equation~\eqref{example_geo} is given, we can also compute the rank over $\overline{\Q}(T)$ with the Shioda-Tate formula~\eqref{rankclos} since for the finite place $v$ of bad reduction, the reduction type is is $I_1$ with $m_v=1$ and for the place $v$ at infinity 
the reduction type is $I_2^*$ with $m_v=7$. Thus one obtains $8-\sum_v (m_v-1) = 8-6=2$. 

\begin{ex}
It is possible to determine explicitly a non-torsion point over $\Q(T)$ also when the roots of $B$ are not rational. For example, let $B(X)=X^2+X+1$ and $C(X)=X^3-4$, and let $\rho$ and $-\rho-1$ be the roots of $B(X)$. Notice that $C(\rho) = -3$ is a square in $\Q(\rho)$. Then the curve
\begin{equation}\label{exemple_1}
E \colon Y^2 = B(X)T + C(X) = X^3+TX^2 + TX+ T-4
\end{equation}
has the two conjugates points $P_1=(\rho,2\rho+1), P_2=(-\rho-1,-2\rho-1)\in\Q(\rho)(T)$ suggested by Theorem~\ref{theoconjclass}. Then, summing them we find the point $P_1+P_2 = (-T+5,2T-11) \in E(\Q(T))$. For more details on the computations, see Section \ref{algrem}. 
\end{ex}

\subsubsection{The case $A$ is a non-zero square}
Let us assume now  that $A(X)=k^2$ is a non-zero square in $\Q$. In this case Theorem~\ref{mtt} gives
\begin{equation}\label{Aequalsquare}
r =  \omega(B^2-4k^2C) -1.
\end{equation}
The maximal possible rank is 3 which is obtained when $B^2-4k^2C$ is a product of 4 distinct linear factors in $\Q[X]$  (and in particular $\deg(B)=2$). Under this condition, we have $B^2(X)-4k^2C(X)=b(X-r_1)(X-r_2)(X-r_3)(X-r_4)$ for some $b\in\Q\setminus\{0\}$, and $r_1,r_2,r_3,r_4\in\Q$ distinct. Then we get four rational points of the form $\big(r_i,kT+\frac{B(r_i)}{2k}\big)$.  The proposition below shows that these points sum to zero and thus are not linearly independent. In particular, this gives a natural interpretation of the "-1" term in formula~\eqref{Aequalsquare}. 
\begin{prop}\label{sum_points}
Let $\E$ as in~\eqref{familyCurve} with $A=k^2$, $k\in\Q\setminus\{0\}$. For every root $\theta \in \overline{\Q}$ of $B^2-4k^2C$ consider the point $P_\theta:=\big(\theta,kT+\frac{B(\theta)}{2k}\big) \in \E(\Q(T))$.  Then, we have 
$$
\sum_{\theta} n_\theta P_\theta = O \in \E(\Q(T)),
$$
where $n_\theta$ is the multiplicity of $\theta$ as a root of $B^2-4k^2C$.
\end{prop}
\begin{proof}
The degree $d$ of the polynomial $B^2-4k^2C$ is 4 or 3 depending on the degree of $B$ ($\deg B= 2$ or $\leq 1$) and in any case the degree of $B^2-4k^2C$ is the degree of the function $Y-kT-\frac{B(X)}{2k}$ over the elliptic 
curve $E/\Q(T)$. The divisor of the function is then 
$$
\sum_{\theta} n_\theta[P_\theta] - d[O]
$$
and the claim follows.
\end{proof}

The proposition above also implies that whenever $B^2-4k^2C$ is the $\ell$-th power of an irreducible polynomial (hence $\ell\in\{3,4\}$ and $r=0$) then the sum of these points without multiplicity is a torsion point of order $\ell$. Notice that the order cannot be a proper divisor of $\ell$ since $P_\theta$ is different from $O$ and cannot have order $2$. This allows to find 
rank $0$ families of elliptic curves over $\Q(T)$ with a rational torsion point.
~\\
For example, we obtain families of curves $\E$ as in \eqref{familyCurve} with $A=k^2$ giving a 3-torsion point by imposing $B^2-4k^2C$ to be equal to $(X-w)^3$ up to a multiplicative constant (so $B$ has degree $\leq 1$). Then, $\E$ is given by the following 3 parameters equation
$$ 
\E \colon Y^2 = X^3 + \frac{48Tb_1k^4 + 24b_0b_1k^2 + b_1^4}{48k^4}X + \frac{1728T^2k^8 + 1728Tb_0k^6 + 432b_0^2k^4 - b_1^6}{1728k^6}
$$
where $b_0,b_1$ and $k$ are rational numbers. The point $\displaystyle P = \left( \frac{b_1^2}{12k^2},\frac{24k^4T+12b_0k^2+b_1^3}{24k^3}\right)$ is a 3-torsion point (this can be independently verified also by specific computations in PARI-GP, see Section~\ref{algrem}).\smallskip
~\\
Similarly, the families  of $\E$ as in \eqref{familyCurve} giving a 4-torsion point are the ones for which $B^2-4k^2C$ is equal to $(X-w)^4$ up to a multiplicative constant, and $\E$ must be given by the following 3 parameters equation
\begin{align*}
\E \colon Y^2 & =  X^3 + Tb_2X^2 + \frac{4k^4-3b_2^2b_1^2+4Tb_2^4b_1}{4b_2^4}X  \\ & + \frac{8k^6 + (-16b_2b_1 + 12Tb_2^3)k^4 + (9b_2^2b_1^2 - 12Tb_2^4b_1 +4T^2b_2^6)k^2 + (-b_2^3b_1^3 + Tb_2^5b_1^2)}{4b_2^6}  
\end{align*}
where $b_2,b_1$ and $k$ are rational. The point $\displaystyle P = \left( \frac{2k^2-b_2b_1}{2b_2^2}, \frac{4k^3+(-3b_2b_1+2Tb_2^3)k}{2b_2^3}\right)$ is a 4-torsion point.

Looking back at Equation~\eqref{Aequalsquare}, we  can also deduce the following curious result.
\begin{corol} Let ${\mathcal C} \colon Y^2=C(X)$ be an elliptic curve defined over $\Q$ with $C$ a monic polynomial of degree 3. Assume that ${\mathcal C}(\Q)$ has rank 0 and that ${\mathcal C}_{\text{tors}} = \{0\}$,
 then for every polynomial $B$ of degree $\leq 2$  and every $k \in \Z^*$ the polynomial $B^2-4k^2C$ is the power of an irreducible polynomial over $\Q(X)$.
\end{corol}
\begin{proof}
If this was false, by~\eqref{Aequalsquare} the curve $\E$ defined by $Y^2=k^2T^2+BT+C$ would be a non-isotrivial elliptic curve over $\Q(T)$ with rank at least 1, but specialization of $T$ at $T=0$ would give a rational point of ${\mathcal C}$. 
\end{proof}

\begin{ex}
 For every $b_1, b_0$ in $\Z$ and $k$ in $\Z\setminus\{0\}$, the polynomial
$$
- 4k^2X^3 + b_1^2X^2 + (2b_0b_1 - 4k^2)X + (b_0^2 - 20k^2)
$$
is the power of an irreducible polynomial since $Y^2=X^3+X+5$ has no rational points except~$0$ (we have chosen $B$ of degree 1). Actually, this polynomial is irreducible, since otherwise it would be a third power, and a quick computation shows that this is not possible since the polynomial is homogeneous in $k, b_0$ and $b_1$ and $-4$ is not a cube in $\Q$. 
\end{ex}

\subsubsection{The case $A$ is a non-square}
Let us assume now  that $A(X)=k$ is not a square in $\Q$. In this case Theorem~\ref{mtt} gives 
\begin{equation}\label{Aequalnonsquare}
r =  \sigma(k,B^2-4AC).
\end{equation}
The definition of $\sigma(\cdot,\cdot)$ implies that $r\leq 2$, and the equality is achieved whenever $B^2-4kC= F_1\cdot F_2$ with $F_1, F_2$ quadratic irreducible (with discriminants $D_{F_1}$ and $D_{F_2}$) such that $D_{F_1}k$ and $D_{F_2}k$ are squares in $\Q$. 
\
\begin{ex} Let us consider 
\begin{equation}\label{example_2}
\E \colon Y^2 = X^3 + TX^2 + (3T-2)X+7T^2-2T
\end{equation}
Then, we have $A=k=7$, $B=X^2+3X-2$ and $C=X^3-2X$. Thus, 
$$
B^2-4AC = X^4 - 22X^3 + 5X^2 + 44X + 4
$$
is an irreducible polynomial of degree 4. If $\theta$ is a root of $P$ then $k=7$ is square in $K=\Q(\theta)$. Indeed, $7=\ell(\theta)^2$ where
$$
\ell(\theta)= \frac{1}{8} \theta^3 - \frac{11}{4}\theta^2 +\frac{7}{8} \theta +\frac{11}{4}.
$$
Furthermore, the field $\Q(\sqrt{7})$ is the only subfield of $K$ of degree 2 and hence the rank of $E$ is~1. 
We obtain the point $P_\theta= \left(\theta, \ell(\theta)T+\frac{B(\theta)}{2\ell(\theta)}\right)$. The sum of the four conjugate points gives
 $$
 \left(\frac{T^4 + 22T^3 + 5T^2 - 12T + 4}{16T^2}, \frac{-T^6 - 41T^5 - 245T^4 + 73T^3 + 42T^2 - 36T + 8}{64T^3}\right) = 2\left(-T,2T\right)
 $$
 in $\Q[T]$.  In fact, if $\theta_1, \theta_2 \in \C$ are roots of $B^2-4AC$ such that $\ell(\theta_1) = \ell(\theta_2)$ then $P_{\theta_1} + P_{\theta_2} = (-T,2T)$ which is easily verified to be of infinite order.\end{ex}

\subsection{The case $A$ has degree $1$ or $2$}\label{Adegresup0}
If $\deg(A)\in\{1,2\}$ then by Theorem~\ref{theoconjclass} we have
\est{
r\leq r_+:=\Omega(M_{(B^2-4AC)^*,A})-\omega(B^2-4AC) -\Upsilon_{B^2-4AC,A}
}
since $\Omega(M_{\gcd(A,B^*),C})\leq 2\deg(\gcd(A,B^*))$, $\omega(\gcd(A,B,C)) \leq2\deg(\gcd(A,B^*,C)) $ and since $\Upsilon_{B^2-4AC,A}=0$ if $\deg(A)=1$.
In the following proposition, we show that $r_{+}  \leq \deg((B^2-4AC)^*)$, so that in particular we have $r\leq 4$ if $A$ has degree 1 and $r\leq 5$ if $A$ has degree $2$. Moreover, these equalities are achieved whenever $B^2-4AC$ splits as a product of 4 (or 5) distinct linear  factors in $\Q[X]$ such that $A(\rho)$ is a square in $\Q$ for every root $\rho$ of $B^2-4AC$.

\begin{prop}\label{propEstimateRank} 
If $\deg(A)=1$ or $2$ then 
$$
 r_{+}  \leq \deg((B^2-4AC)^*).
 $$
 Furthermore, we have $r_{+}=\deg((B^2-4AC)^*)$ if and only if $(B^2-4AC)^*$ is a product of irreducible polynomials over $\Q$ of degree 1  such that $A(\rho)$ is a square for all roots $\rho$ of $(B^2-4AC)^*$. 
\end{prop}
\begin{proof}
By definition we have 
\est{
r_+=\sum_{F|B^2-4AC,\atop \text{$F$ irr.}}(\Omega(M_{F,A})-1 -\Upsilon_{F,A}).
}
Clearly if $\deg(F)=1 $, then $\Omega(M_{F,A})-1 -\Upsilon_{F,A}=\Omega(M_{F,A})-1 \leq1$. Furthermore, the equality is achieved if and only if $A(\rho)$ is a square where $\rho$ is the root of $F$. 
In particular, the proposition follows if we can show that $\Omega(M_{F,A})-1 -\Upsilon_{F,A}<\deg(F)$ for $2\leq \deg(F)\leq 5$.  
In order to see this, we first notice that by Lemma~\ref{factors} we have $\Omega(M_{F,A})\in\{1,2,4\}$ if $\deg(F)\in\{2,4\}$ and $\Omega(M_{F,A})\in\{1,2\}$ if $\deg(F)\in\{3,5\}$. This readily implies that $\Omega(M_{F,A})-1 -\Upsilon_{F,A}<\deg(F)$ unless $\deg(F)=2$ and $\Omega(M_{F,A})=4$. By Lemma~\ref{factors} this happens whenever $F=P(A)$ with $P$ of degree 1 (and so $A$ must be of degree 2) and $A(\rho)$ a square in $\Q(\rho)$ for any roots $\rho$ of $F$. In this case then the Euclidean division of $A$ by $F$ has the form $A=\lambda F + c^2$, with $\lambda, c\in\Q, \lambda\neq0$ and thus, by~\eqref{dfups}, we  have $\Upsilon_{F,A}=2$. In particular, $\Omega(M_{F,A})-1 -\Upsilon_{F,A}=1<\deg(F)$, as desired.
\end{proof}

As mentioned above, we achieve the maximal rank $r=5$ if $\deg(A)=2$ and $B^2-4AC$ splits in $\Q[X]$ as a product of 5 distinct linear factors all of  whose roots $\rho$ are such that $A(\rho)$ is a square in $\Q$. We shall now show how to construct several families satisfying these conditions.

Up to a linear change on the variable $X$, we can assume that $B^2-4AC$ has the form $\ell X(X-1)(X-r_3)(X-r_4)(X-r_5)$. Furthermore $A(0)$ and $A(1)$ must be squares and for simplicity we assume that 
$A(X)=X(X-1)+ k^2$ for some $k\in \Z$. Also, we write 
\begin{align*}
B(X)&=b_2X^2+b_1X+b_0 ; \\
C(X)&= X^3+c_2X^2+c_1X+c0; \\
\Delta(X) &= \ell  X(X-1)(X-r_3)(X-r_4)(X-r_5).
\end{align*}
We now want to solve $P=\Delta - (B^2-4AC) = 0$. We can do this via the following steps. 

\begin{itemize}
\item We ensure that the degree 5 term in $X$ of P is zero by taking $\ell =-4$.
\item We choose the coefficients $c_2$, $c_1$ and $c_0$ in order to delete the coefficients of $X^4$, $X^3$ and $X^2$ (at this step $b_2$ has to be different from $-b_1$).
\item We choose the coefficient $b_0$ in order to delete the coefficient of $X$.
\end{itemize}
After these steps, we are left with 
\begin{scriptsize}
\begin{align*}
2(b_1+b_2)B(X)&= (2b_2^2 + 2b_1b_2)X^2 + (2b_1b_2 + 2b_1^2)X +{}\\
& \quad +\big(-4k^4 + (2b_2^2 + 2b_1b_2 + (4r_4 + 4r_5 - 4)r_3 + (4r_5 - 4)r_4 -4r_5 + 4)k^2 
 - (b_2^2 + 2b_1b_2 + b_1^2)\big), \\
4(b_1+b_2)C(X)&=4(b_1 + b_2)X^3 + (b_2^3 + b_1b_2^2 -4 (r_3 +r_4 +r_5)(b_2 + b_1))X^2+{}\\
&\quad + \big(-4(b_1+b_2)k^2 + b_2^3 + 3b_1b_2^2 + 2b_1^2b_2+4( r_3r_4 + r_3r_5 + r_4r_5)(b_1+b_2)\big)X +{}\\
& \quad-4b_2k^4 + (b_2^3 + b_1b_2^2 + 4(r_3r_4 + r_3r_5 + r_4r_5)b_2 + 4(r_3 + r_4 + r_5 - 1)b_1)k^2+{}\\
&\quad  + b_1b_2^2 + 2b_1^2 b_2 + b_1^3 - 4r_3r_4r_5(b_1+b_2).
\end{align*}
\end{scriptsize}%
Moreover, to impose the condition that $A(r_3)$, $A(r_4)$ and $A(r_5)$ must be squares, we take $r_i$ of the form $(2m_ik+1)/(1-m_i^2)$ for $m_i \in \Q$ so that $A(r_i)=s(r_i)^2$ with
$$ s(r_i):=\frac{(1+m_i^2)k+m_i}{1-m_i^2}.$$
Substituting these values of $r_i$ in $P$, we see that $P$ factors as $P=P_1\cdot P_2$ for some (long) polynomial expressions $P_1$ and $P_2$
 (the factorization was done in Magma). We now make the simplifying assumption $m_4=1/m_3$ which allows for a further factorization of $P_1$ into $(m_5+1)\tilde{P_1}$ with $\tilde{P_1}$ of degree 1 in $m_5$. We choose $m_5$ to be the solution of $\tilde P_1$, that is $m_5= \frac{q-p}{q+p}$ with %
 \begin{scriptsize}
\begin{align*}
p&= 4(m_3^4 +2m_3^2 +1)k^4 +8 (m_3^3 +m_3)k^3 + 4m_3^2k^2 \\
q&=-2 \big((m_3^4 -2m_3^2 +1)b_2^2 +2 (m_3^4 -2m_3^2 +1)b_1b_2 + (m_3^4 -2m_3^2 +1)b_1^2\big)k+{} \\
&\quad+ (m_3^4 - 2m_3^2 + 1)b_2^2 + 2(m_3^4 - 2m_3^2 + 1)b_1b_2 + (m_3^4 - 2m_3^2 + 1)b_1^2.
\end{align*}
\end{scriptsize}%
obtaining the desired equality $P=0$. Thus, we obtain the 4-parameters elliptic curve over $\Q(T)$ given by
\begin{scriptsize}
\begin{align*}
A(X)&=X(X-1)+k^2,\\
2(b_1+b_2)(m_3^2-1)^2B(X)&=2\big((m_3^4 - 2m_3^2 + 1)b_2^2 + (m_3^4 - 2m_3^2 + 1)b_1b_2\big)X^2 + 2\big((m_3^4 - 2m_3^2 + 1)b_1b_2 + (m_3^4 - 2m_3^2 + 1)b_1^2\big)X+{}\\
&\quad -4(m_3^4 +2 m_3^2 +1)k^4 -8 (m_3^3 +m_3)k^3 + 2\big((m_3^4 - 2m_3^2 + 1)b_2^2 + (m_3^4 - 2m_3^2 + 1)b_1b_2 - 2m_3^2\big)k^2 +{}\\
&\quad - \big((m_3^4 - 2m_3^2 + 1)b_2^2 + 2(m_3^4 - 2m_3^2 + 1)b_1b_2 + (m_3^4 - 2m_3^2 +1)b_1^2\big),\\
C(X)&= \frac{B(X)^2+4X(X-1)(X-r_3)(X-r_4)(X-r_5)}{4A(X)},
\end{align*}
\end{scriptsize}%
where $r_3=\frac{2m_3k+1}{1-m_3^2}$, $r_4=1-r_3$ (this follows from the choice $m_4=1/m_3$) and $r_5=\frac{2m_5k+1}{1-m_5^2}$. By construction, for generic choices of the parameters the resulting elliptic curve has rank $5$ over $\Q(T)$. Moreover, one can also explicitly give 5 rational points:
\est{
&P_1=\left(0,kT+\frac{B(0)}{2k} \right), \quad
P_2=\left(1,kT+\frac{B(1)}{2k} \right), \quad
P_3=\left(r_3, s(r_3)T + \frac{B(r_3)}{2s(r_3)}\right), \\
&P_4=\left(r_4, s(r_4)T + \frac{B(r_4)}{2s(r_4)}\right), \quad
P_5=\left(r_5, s(r_5)T + \frac{B(r_5)}{2s(r_5)}\right).
}
We have not tried to prove that these 5 points are independent in general. However, we verified that when we specialize at the values $b_2=b_1=k=T=1$ and $m_3=2$, the points above are independent on the corresponding elliptic curve over $\Q$.

\section{Families of different type}\label{difftype}
The techniques employed in this work can provide results also for families of elliptic curves different from those described in the previous sections. Indeed, one can also study some families with $\deg(A)=3$ allowing for the larger rank $r=6$, and generalizing work of~\cite{arms_all} who constructed some special families of this shape. One can also study the rank of twists of~\eqref{familyCurve} recovering and generalizing results in~\cite{bettinDavidDelaunay}, see Remark~\ref{rmktwist}.

\subsection{Families of rank 6} 
Let us consider the curve defined over $\Q(T)$ as
\begin{equation}\label{rank6curve}
    \mathcal{C}: Y^2 = A(X)T^2 + B(X)T + C(X)
\end{equation}
with $A=L^3$, $\deg L=1$ and $\deg B, \deg C\leq 3$ (now the polynomial $C$ is not necessarily monic). Curves of this shape were considered in \cite{arms_all}.
Denoting by $a,b$ and $c$ the leading coefficient of $L,B$ and $C$ respectively and reversing the role of $X$ and $T$, we see that the curve $\mathcal{C}$ can be written as
\begin{align*}
    Y^2 = (a^3T^2+bT+c) X^3+p_2(T) X^2+p_4(T)X+p_6(T),
\end{align*}
with $\deg_T(p_i)=2$ for $i=2,4,6$. Multiplying both sides by $(a^3T^2+bT+c)^2$ and setting the change of variables $(a^3T^2+bT+c)y\to y $, $(a^3T^2+bT+c)X\to X$, the curve is now defined by the equation in Weierstrass form
\begin{equation*}
    Y^2 = X^3 + q_2(T)X^2 + q_4(T)X+q_6(T),
\end{equation*}
with $\deg_T (q_i)\leq i$ for $i=2,4,6$: this condition assures that $\mathcal{C}$ is a rational elliptic surface, and thus we can compute its rank $r$ over $\Q(T)$ using Nagao's formula. In particular, $r$ can be computed using~\eqref{msfr}.

\begin{theo}
Let $\mathcal{C}$ be a family of elliptic curves as in \eqref{rank6curve}, and let $r$ be its rank over $\Q(T)$. Then 
\begin{equation}\label{degree1rank6}
r = \Omega(M_{(B^2-4AC)^*,L})- \omega(B^2-4AC) -\Xi_{L,B,C}.
\end{equation}
where $\Xi$ is as in~\eqref{dfxi}.
\end{theo}

\begin{proof}
The proof is essentially identical to that of Theorem~\ref{mtt}, since in this case~\eqref{msfr} can be rewritten as 
\begin{align*}
r&=\lim_{Q\to\infty}\frac{\log Q}Q\sum_{p\leq Q}(S_{\gcd(L,B),C}(p)+S_{B^2-4AC,L}(p))-\square(L).\qedhere
\end{align*}
\end{proof}
\noindent
Just like for Proposition \ref{propEstimateRank}, we can employ Lemma \ref{factors} to obtain an upper bound for $r$.

\begin{corol}
Let $\mathcal{C}$ be as in \eqref{rank6curve}. Then
$r\leq \deg((B^2-4AC)^*)-\Xi_{L,B,C}. $ 
Furthermore, the equality holds if and only if $(B^2-4AC)^*$ is a product of degree $1$ polynomials over $\Q$ such that $L(\rho)$ is a square for all roots $\rho$ of $(B^2-4AC)^*$.  
\end{corol}
\noindent
In particular, if $B$ and $C$ are cubic polynomials, we can obtain families of elliptic curves with rank 6 whenever $B^2-4AC$ splits as a product of 6 coprime linear factors over $\Q$ and $L(\rho)$ is a square for every root $\rho$ of $B^2-4AC$. This recovers and generalizes some results in \cite{arms_all}.

\begin{ex} 
Let us consider the polynomials
\begin{align*}
A(X)&=X^3,\\
B(X)&=b_3X^3+b_2X^2+b_1X+b_0 , \\
C(X)&= c_3X^3+c_2X^2+c_1X+c_0, \\
\Delta(X) &= \ell  (X-r_1^2)(X-r_2^2)(X-r_3^2)(X-r_4^2)(X-r_5^2)(X-r_6^2).
\end{align*}
with the numbers $r_i$ being positive and distinct integers and $b_1,b_2,b_3,c_1,c_2,c_3\in\Q$ with $b_3\neq0$.
It is then possible to solve $P=\Delta-(B^2-4AC)=0$ by applying linear substitutions which determine the values of $\ell, b_2, b_1$ (with $\ell\neq0$) and all the parameters of $C$, with $b_3$, $b_0$ and $r_i$ as remaining free 8 parameters with the following 6 rational points 
\begin{equation*}
    P_i =\left(r_i, r_i^3 T + \frac{B(r_i)}{2 r_i^3}\right),\hspace{0.2cm}i=1,\ldots,6.
\end{equation*}

As an example, if we choose the values $r_i=i$, we obtain a family described by the polynomials
\begin{scriptsize}
\begin{align*}
B(X)&=b_3X^3 + \frac{9919}{1280000} b_0X^2 - \frac{5369}{7200} b_0 X + b_0,\\
C(X)&=\left(\frac{1}{4}b_3^2 - \frac{1}{2073600}b_0^2\right)X^3 + \left(\frac{9919}{2560000}b_0b_3 + \frac{91}{2073600}b_0^2\right)X^2 + \left(-\frac{5369}{14400}b_0b_3 - \frac{253599562853}{176947200000000}b_0^2\right)X \\
&+ \left(\frac{1}{2}b_0b_3 + \frac{3078544001}{165888000000}b_0^2\right).
\end{align*}
\end{scriptsize}%
Then the resulting families have all ranks $6$ and have the following 6 rational points
\begin{equation*}
    P_i =\left(i, i^3 T + \frac{B(i)}{2 i^3}\right),\hspace{0.2cm}i=1,\ldots,6.
\end{equation*}
\end{ex}

\section{Algorithmic remarks}\label{algrem}

\subsection{Constructing rational points}

Theorem~\ref{mtt} added with Lemma~\ref{factors} allows one to find explicit rational points of $E(\Q(T))$ as shown in our previous examples. In this section, 
we explain the construction with an algorithmic point of view. We wrote  PARI-GP scripts to compute ranks and points and find the examples contained in the paper (the scripts are available at \url{http://delaunay.perso.math.cnrs.fr/publications.html}). 

\subsubsection{The case $A=0$}

In this case, the rank comes from (the conjugacy classes of) roots $\theta$ of the polynomial $B=b_2X^2+b_1X+b_0$ such that $C(\theta)$ is a square in $\Q(\theta)$. \smallskip
~\\
If such a $\theta$ is in 
$\Q$, with $C(\theta)= c^2 \in \Q,$ then the point $P=(\theta,c) \in \E(\Q(T))$ where the computation of $c$, if it exists, can be done directly.\smallskip
~\\
If such a $\theta$ is not in $\Q$, then it belongs to a quadratic number 
field $K$ defined by $B$. One can detect if $C(\theta) = r\theta + s$ is a square in $K$ by solving the system of two equations $(x\theta+y)^2 = r\theta + s$ with unknown $x,y \in \Q$ (it consists in finding rational roots of a degree 4 
polynomial with coefficient in $\Q$) or by factoring $X^2-C(\theta)$ in $K$ as it is explained in \cite{cohen0} (and implemented in \cite{pari}). Now, the conjugate $\theta^c$ of $\theta$  belongs to $K$ with 
$\theta^c = -\theta - \frac{b_1}{b_2}$. If $C(\theta) = c^2 = c(\theta)^2$, the square roots of $C(\theta)$ are expressed as a degree one polynomial in $\theta$ and then $P_\theta= (\theta,c(\theta))$; hence $P_{\theta^c}$ 
belongs to $\E(K)(T)$ and is given by polynomials in $\theta$. The trace $P_\theta + P_{\theta^c}$ gives a rational point in $\E(\Q(T))$. An example of such a construction is given by the elliptic curve defined 
by equation~\eqref{exemple_1}.

\subsubsection{The case $A \neq 0$}\label{mascot}

In this case, as it can been seen in Theorem~\ref{mtt} and in Lemma~\ref{factors}, the rank (in fact $r_+$) increases with the number of (conjugacy classes of)  roots, $\theta \in \overline{\Q}$
of $B^2-4AC$ such that $A(\theta)$ is a square in $\Q(\theta)$.\\
Let $F$ be an irreducible factor of $B^2-4AC$ and $\theta \in \overline{\Q}$ one of its roots. Again, one can detect if $A(\theta)$ is a square in $K=\Q(\theta)$
by factorizing the polynomial $X^2- A(\theta)$ in $K$ as in \cite{cohen0}. If this is the case, we obtain a polynomial $a \in \Q[X]$ such that $A(\theta) = a(\theta)^2$. 
Hence, the point $P_\theta= \left(\theta, a(\theta)T + \frac{B(\theta)}{2a(\theta)}\right)$ 
belongs\footnote{If $A(\theta)=0$ then we must have $B(\theta)=0$ and the contribution in $\gcd(A,B)$ has a negative effect on the rank. However, even in that case, we can look at wheter $C(\theta)$ is a square and consider the point 
$(\theta,c(\theta))$ where $c(\theta)^2 = C(\theta)$.} 
to $\E(K)(T)$. We then need to compute the trace of this point.\smallskip
~\\
Notice that if we are in this situation then $\deg F \leq 5$. In particular, the degree is small enough to compute the Galois group and the Galois action on the conjugate of the 
points and to perform the exact computation of the trace as an algebraic expression of $\theta$. An example of such a construction is given by the elliptic curve defined by equation~\eqref{example_2} where in fact the sum of 
2 points is sufficient to find a rational point. This seems to be a general fact: when the degree of $F$ is 4 and $a(\theta)$ generates a degree two subfield, then the sum of two well chosen points is sufficient to obtain a 
rational point. We did not try to prove it but it should come from Galois properties (in these cases the Galois closure of $K$ is $K$ itself or a dihedral field of degree 8).  \smallskip
~\\
A second more general and efficient process was explained to us by Nicolas Mascot to whom we are very grateful. \smallskip
~\\
Let $P_1=(x_1,y_1), P_2=(x_2,y_2), \cdots, P_n=(x_n,y_n)$ be the $n$-conjugates of an algebraic point $P$, with $n\geq 2$. Let $L(X)$ be the Lagrange interpolator of degree $\leq n-1$ such that 
$L(x_i)=y_i$ ($1\leq i \leq n$). Consider the function $F := Y-L(X)$ on the elliptic curve: it has degree at most $ 2n-2$, only one pole $0$ and the points $P_i$ are zeros of $F$. Hence its 
divisor is $-(\deg F)[O] + \sum_{i=1}^n [P_i] + \sum_{j=1}^\ell [Q_j]$ for certain convenient points $Q_j=(x'_j,y'_j)$ ($1 \leq j \leq \ell$).\\
If we replace $Y$ by $L(X)$ in the equation defining the elliptic curve, 
we obtain a polynomial in $X$ whose roots are the $x_i$ and the $x'_j$. Since we can divide this polynomial by the characteristic polynomial of $x_i$, we obtain the conjugates $x'_j$ and 
finally the $y'_j$ by $y'_j=L(x'_j)$. Now, the number $\ell$ of points $Q_j$ is $\ell = \deg F - n \leq n-2 \leq n$.  Since the divisor of $F$ is principal, we have $\sum_{i=1}^n P_i = \sum_{j=1}^\ell Q_j$ and then we can 
apply a recursive computation: if $n=2$ then the degree of $F$ is 3 (coming from the degree of $Y$) so that $\ell=1$ and the point $Q_1$ is rational. 

\begin{ex} 
Let us consider
\begin{equation}\label{example_3}
\E \colon Y^2= X^3 + (T^2 - 11T - 1)X^2 + (-24T^2 + 52T + 5)X + (144T^2 -48T - 67)
\end{equation}  
We have $A(X)=X^2 - 24X + 144=(X-12)^2$, $B=-11X^2 + 52X - 48$ and $C= X^3 - X^2 + 5X - 67$. With our formula, we obtain that the rank of $\E$ is 2. We have 
\begin{eqnarray*}
B^2-4AC &=&  -4X^5 + 221X^4 - 1836X^3 + 5084X^2 - 14304X + 40896\\
&=&-(X-4)(X^2 - 52X + 284)(4X^2 + 3X + 36).
\end{eqnarray*}
Each factor gives a rational point because $A$ is a square and so $A(\theta)$ is a square in $\Q(\theta)$ for every root of $B^2-4AC$. The first factor leads to the point $P_1=(4,8T-1)$, the second factor leads to 
$$
P_2= \left(\frac{-36}{7}T + \frac{2773}{196}, -\frac{36}{7}T^2 + \frac{8557}{196T} - \frac{140869}{2744}\right)
$$
and the third one to 
$$
P_3= \left(\frac{56}{9}T + \frac{604}{81}, \frac{56}{9}T^2 - \frac{524}{27}T - \frac{13229}{729}\right) .
$$
Those points are dependent since $P_1+P_2+P_3 = O$ (the proof of Proposition~\ref{sum_points} can be generalized to this case since $A$ is a square).
\end{ex}
\subsection{Detecting rational points on elliptic curves defined over $\Q$}
    
One can use Theorem~\ref{mtt} and the section above as an alternative to direct brute force point searching on elliptic curves defined over $\Q$. For this purpose, let 
$$
{\mathcal C} \colon y^2 = C(x) = x^3+c_2x^2 +c_1x+c_0
$$
be an elliptic curve defined over $\Q$. In order to find points in ${\mathcal C}$ one can look for polynomials $A$ and $B$ of degree $\leq 2$ such that the rank of 
$$
\E \colon Y^2 = A(X)T^2 + B(X)T + C(X) 
$$
is greater than 1. Then finding rational points on $\E(\Q(T))$ will lead to rational points on ${\mathcal C}$ by specialization of $T=0$. 
\begin{ex}
 Let ${\mathcal C} \colon y^2= x^3+43x+30$. The analytic (and thus the algebraic) rank of ${\mathcal C}$ is 1. Searching through polynomials $A(X),B(X)$ with bounded coefficients, we obtain that $\E$ has rank 1 for example for $A(X)=0$ and $B(X)=X^2-49X+149$. Also, we find the corresponding point $\left(-T +\frac{65}{16}, \frac{15}{4}T-\frac{1055}{64}\right)\in \mathcal C$. Hence, we obtain the point 
$$
\left(\frac{65}{16},-\frac{1055}{64} \right) \in {\mathcal C}(\Q)
$$
of height $\approx 5$ (this point was found in less than 1 second with this process). \smallskip
\end{ex}

Another application is the following. Assume that ${ \mathcal C}$ has rank $\geq 2$ over $\Q$. Then, we can make a brute search on $A$ and $B$ so that the rank of $ \E \colon Y^2 = A(X)T^2 + B(X)T + C(X)$ over 
$\Q(T)$ is $\geq 2$.  Assume that a point $(x_0,y_0) \in {\mathcal C}(\Q)$ is already known: then, one can make a brute search on $A$ and $B$ with bounded coefficients of the form\footnote{This specific choice seems to be more efficient than to take $A$ of the form $(X-x_0)(a_1X+a_0)$ and $B$ of the form $(X-x_0)(b_1X+b_0)$.} $A(X)=(a_1X+a_0)^2$ and $B(X)= (X-x_0)(b_1X+b_0) + (a_1x_0+a_0)y_0$ and such  that the rank 
of $\E $ is greater than or equal to 2. In this case, the point $(x_0,(a_1x_0+a_0)T+y_0)$ is already a point on $\E(\Q(T))$ leading to $(x_0,y_0)$ on ${\mathcal C}(\Q)$ whereas the specialization of other points on $\E$ typically gives a new rational point on ${\mathcal C}$.
\begin{ex} 
Let ${\mathcal C} \colon y^2 = x^3 - x^2 + 5x - 67$. It has (analytic) rank 2 over $\Q$. The point $P_1=(4,1)$ can be easily obtained. Then applying the method above, we find the elliptic curve defined over $\Q(T)$ with 
equation \eqref{example_3}. This leads to a second point 
$$
P_2=\left(\frac{2773}{196}, - \frac{140869}{2744}\right).
$$
The regulator $R$ of $(P_1,P_2)$ is $\approx 12.65$. We notice that, up to the accuracy of the computation, $\frac{L''(\mathcal C,1)}{2!} \approx \Omega R$, where $\Omega$ is the real period of $E$ and $L$ is the L-function of $\mathcal C$. Since the product of the 
Tamagawa numbers is $1$ and $\mathcal C(\Q)$ is torsion free, this suggests that the  Tate-Shafarevich group of $\mathcal C$ is trivial if $(P_1,P_2)$ is a basis of $\mathcal C$. 
\end{ex}

\bibliographystyle{alpha}
\bibliography{strange_2}

\end{document}